\documentclass[runningheads]{llncs}
\usepackage{graphicx}
\usepackage{url}
\usepackage{cite}
\usepackage{amsmath}
\usepackage{amssymb,amscd} 
\usepackage{enumerate} 
\usepackage{algorithm}
\usepackage{algpseudocode}

\usepackage{siunitx}
\newcommand{\RR}{\mathbb{R}}

\newcommand{\dd}{\, {\rm d}}
\newcommand{\NN}{\mathbb{N}}

\def\qset{\mathcal Q}

\def\states{\mathcal X}

\def\cset{\mathcal M}

\def\allgambles{\RR^m}

\def\lpr{\underline P}
\def\upr{\overline P}

\def\gambleset{\mathcal F}
\def\metherr{E^m}

\newcommand{\low}[1]{{\underline{#1}}}
\newcommand{\up}[1]{{\overline{#1}}}

\begin{document}
\title{Computing bounds for imprecise continuous-time Markov chains using normal cones}
\titlerunning{Computing bounds for CTIMC}
%
\author{Damjan \v{S}kulj\inst{1}\orcidID{0000-0002-6177-585X} }
\authorrunning{D. \v{S}kulj}
%
\institute{ University of Ljubljana, Faculty of Social Sciences \\ Kardeljeva pl. 5, SI-1000 Ljubljana, Slovenia\\
\email{damjan.skulj@fdv.uni-lj.si}}
\maketitle              
\begin{abstract}
	The theory of imprecise Markov chains has achieved significant progress in recent years. 
	Its applicability, however, is still very much limited, due in large part to the lack of efficient computational methods for calculating higher-dimensional models. 
	The high computational complexity shows itself especially in the calculation of the imprecise version of the Kolmogorov backward equation. 
	The equation is represented at every point of an interval in the form of a minimization problem, solvable merely with linear programming techniques. 
	Consequently, finding an exact solution on an entire interval is infeasible, whence approximation approaches have been developed. 
	To achieve sufficient accuracy, in general, the linear programming optimization methods need to be used in a large number of time points. 
	
	The principal goal of this paper is to provide a new, more efficient approach for solving the imprecise Kolmogorov backward equation.
	It is based on the Lipschitz continuity of the solutions of the equation with respect to time, causing the linear programming problems appearing in proximate points of the time interval to have similar optimal solutions. 
	This property is exploited by utilizing the theory of normal cones of convex sets. 
	The present article is primarily devoted to providing the theoretical basis for the novel technique, yet, the initial testing shows that in most cases it decisively outperforms the existing methods.

	\keywords{Imprecise Markov chain in continuous-time  \and Imprecise transition operator \and Normal cone.}
\end{abstract}
	\section{Introduction}
	The theory of imprecise Markov chains in continuous-time has achieved significant progress in recent years \cite{de2017limit, erreygers2017imprecise, erreygers2018computing, skulj:15AMC}, following the success of imprecise Markov chains in discrete time \cite{decooman-2008-a, skulj:09ISIPTA}. They successfully combine the theory of stochastic processes with the ideas of imprecise probabilities \cite{augustin2014introduction, walley:91}. The theory has been employed in the analysis of optical networks \cite{erreygers2018imprecise, rottondi2017modelling}, electric grid \cite{troffaes2015using, troffaes2019two}, and information propagation \cite{liu2020double}. 
	
	The applicability of the theory is still limited to cases with moderate number of states, mainly because of the computational complexity. The core of the computations with imprecise (and precise) continuous-time Markov chains is the evaluation of the  Kolmogorov backward equation. It is a matrix differential equation, which in the imprecise case involves lower transition operators instead of fixed matrices that are used in the precise theory. Consequently, the closed-form expressions known from the precise case are unfeasible for the imprecise model. As an alternative, numerically intensive grid methods have been developed \cite{krak2017imprecise, erreygers2017imprecise}. Those divide the interval of interest into a large number of subintervals where an optimization problem is solved using linear programming techniques. An additional difficulty is that the problem, in general, could hardly be tackled with parallel computation, as the outputs from earlier parts of the interval serve as the inputs for those coming later. 
	
	An alternative approach has already been presented in \cite{skulj:15AMC}, with a hybrid method. The method combines the matrix exponential approach, known from the precise case, and grid techniques, in the situation where the matrix exponential approach is infeasible. This proposal seems to have been overlooked in later papers on the topic \cite{erreygers2017imprecise, krak2017imprecise}, which only focus on the improvements of the grid technique. The reason may be that the hybrid method originally presented is not fully optimized for practical use. 
	
	The goal of the present article is to fill this gap and provide a computationally efficient algorithm based on the idea proposed in \cite{skulj:15AMC}. To make the method more suitable for practical use, we combine it with the theory of normal cones of convex sets. It allows substituting several steps that were initially based on linear programming with computationally simpler matrix operations. The primary result proposed is a computationally efficient procedure for solving the imprecise version of the Kolmogorov backward equation. It proceeds by identifying intervals where a solution using a suitable matrix exponential produces sufficiently accurate approximations within given error bounds. In most cases, the intervals allowed by our approach are considerably larger than those required by the existing grid methods. Merely in the worst cases, which are borderline situations typically only restricted to smaller parts of the domain, the interval widths are of about similar sizes. 	
	The identification of the intervals where the matrix exponential method is feasible does bring some additional computational costs to each step. Nevertheless, these computations are in the form of matrix operations and therefore much faster by than the linear programming optimization, which is in general inevitable at each step and still contributes the majority of the computational costs.  	
	
	We illustrate our method by two examples. 
	In our first example, the solution that would require more than a thousand steps with the grid methods, completes in only three steps with our approach. 
	In the second example we formally confirm, in a reasonable number of steps, the validity of a solution from a previous study, where the existing methods were reported as infeasible. 
	The intention of this paper, however, is to provide the theoretical basis for the method and leave the practical considerations to further research. 
	This also includes comparison with the existing methods, as not much practical testing has been reported in literature up to now.

	Our paper is structured as follows. In Section~\ref{s-imcct} we provide an overview of the theory of imprecise Markov chains in continuous-time. In Section~\ref{s-nmfle} essential methods are presented for calculating lower expectations with respect to imprecise probabilistic models. The convexity properties of imprecise transition rate operators and their normal cones are presented in Section~\ref{s-nciq}, and in Section~\ref{s-nqm} the norms and seminorms used throughout the paper are provided. In Section~\ref{s-nmbc} the numerical approximation techniques are discussed and the novel approach is proposed in detail. All mentioned methods are analyzed from the point of view of errors they produce in Section~\ref{s-ee}. Finally, in Section~\ref{s-ae} the proposed methods are merged into a working algorithm and demonstrated on two examples. 
	
	\section{Imprecise Markov chains in continuous-time}\label{s-imcct}
	\subsection{Imprecise distributions over states}
	An imprecise Markov chain in continuous-time is a stochastic process with a state space $\states$, whose elements will be denoted by $k \in \states$ and its cardinality $|\states|$ by $m$. The states will simply be labelled by consecutive numbers $1, 2, \ldots, m$. Labels however will not have any meaning for the dynamics of the process. The process will be indexed by time $t\in [0, \infty)$. At every time point $t$, the state the process assumes is denoted by $X_t$, which is thus a random variable on $\states$. As we will only consider the finite state case, the measurability considerations will be trivial. The distribution of $X_t$ is assumed to be imprecisely specified, and is therefore represented by an imprecise probabilistic model. The usual choice of the model in the theory of imprecise probabilities are \emph{credal sets} and derived models of \textit{coherent lower} and \textit{upper previsions}. 
	
	Credal sets are closed convex sets of probability distributions or expectation functionals -- depending on how they are presented. A credal set can be represented by listing the extreme points or via constraints in terms of linear inequalities. Because of their large number, growing rapidly with increasing time $t$, listing the extreme points is impractical. Instead, the alternative approach utilizing constraints in terms of linear inequalities proves more efficient. 
	
	In the terminology commonly used in the theory of imprecise probabilities, the constraints are known as coherent lower and upper previsions. They are defined on sets of \emph{gambles}, where a gamble is a common term which in the theory of imprecise probabilities denotes an uncertain reward on $\states$. Technically, a gamble is a real valued map $f\colon \states \to \RR$, which in general is required to be bounded and measurable with respect to some algebra $\mathcal A\subseteq 2^\states$. For finite spaces $\states$ boundedness and measurability are automatically satisfied, which allows identifying the set of all gambles $\mathcal L(\states)$ with the linear space of all real $|\states|$-tuples, or as usually denoted, $m$-tuples. Thus, we will identify $\mathcal L(\states)$ with $\RR^m$. Given a gamble $f\in\allgambles$, $f_i$ will denote its $i$-th component, or $f_i = f(i)$ for every $i \in \{ 1, \ldots,  m\}$. For a subset $A\subseteq \states$ we denote with $1_A$ its \textit{indicator gamble}
	\begin{equation*}
		1_A(i) = \begin{cases}
			1 & i\in A, \\
			0 & i\not\in A.
		\end{cases}
	\end{equation*}
	Particularly, $\lambda 1_\states$, for some $\lambda\in \RR$, is just a constant gamble on $\states$ mapping each $i\in \states$ to $\lambda$.
	
	To explain very briefly, given a set $\gambleset$ of gambles, lower and upper previsions denote a pair of mappings $\lpr, \upr \colon \gambleset \to \RR$ such that $\lpr (f) \leqslant \upr (f)$, which may serve as constraints to forming a credal set of the form
	\begin{equation}
		\cset = \{ P \colon \lpr(f) \leqslant P(f) \leqslant \upr (f) \},
	\end{equation}
	where $P$ stands for \textit{linear expectation functionals} or \emph{linear previsions}. The distinction between both notions is only meaningful in the case of infinite state spaces. Instead of a pair of lower and upper previsions, it is more common to only specify either of them. More on the representation and results on the correspondence between credal sets and lower and upper previsions a keen reader is kindly referred to general literature on imprecise probabilities, such as \cite{augustin2014introduction, miranda:07, walley:91}. The aspects needed for our case will be detailed in the sequel of this manuscript.  
	
	 Credal sets give rise to \emph{lower} and \emph{upper expectation functionals} on the space $\allgambles$ of all gambles on $\states$. Given a gamble $f\in \RR^m$, we define its lower and upper expectation with respect to a credal set $\cset$ as
	 \begin{equation}\label{eq-lower-exp}
	 	\low E (f) = \inf_{P\in \cset} P(f) = \inf_{P\in \cset} \sum_{k\in\states} P(1_{\{k\}}) f(k)
	 \end{equation}
	 and 
	 \begin{equation}\label{eq-upper-exp}
	 	\up E (f) = \sup_{P\in \cset} P(f) = \sup_{P\in \cset} \sum_{k\in\states} P(1_{\{k\}}) f(k)
	 \end{equation}
	 respectively. 
	 If $\cset$ is a credal set corresponding to some lower or upper prevision or a combination of both, then the lower and upper expectations obtained in this way are said to be the \emph{natural extension} of the original assessments. This is because, the assessments on $\gambleset\subset \RR^m$ are extended to the entire space allowing for the maximal set of compatible probability models.
	 
	 The basic properties of lower and upper expectation functionals imply the conjugacy relation $\up E(f) = -\low E(-f)$, meaning that every upper expectation can be deduced from a lower expectation and vice versa, rendering the models equivalent. Thus, only one of the definitions \eqref{eq-lower-exp} and \eqref{eq-upper-exp} is sufficient to completely describe an imprecise probability model. Indeed, in the literature on imprecise stochastic models, both models have been utilized, depending on their convenience in particular cases. Particular notions and formulas can benefit from one or another convention, yet they can very easily be reworked into the conjugate terms. In this paper we follow the approach utilized in the recent papers \cite{de2017limit, erreygers2017imprecise, erreygers2018computing, DECOOMAN201618, krak2017imprecise}, that use lower expectations as the basic model. This is in contrast with some prior papers on stochastic processes in discrete time \cite{decooman-2008-a,skulj:09IJAR,skulj:hable:13MET,Crossman:2010,skulj:13LAA}, where upper expectations were used, which was also the case in our initial approach \cite{skulj:15AMC}. 
	 The essential benefit of using lower expectations compared to extreme points of credal sets is that imprecise probability models manifested in sets of probability models are represented by more tractable real-valued maps. 
	 
	 Adding the time dimension, our analysis now translates into finding the lower expectations $\low E_t(f)$ for a given gamble $f$ with respect to the corresponding credal sets $\cset_t$ at given time $t$. This results in a real valued map $t\mapsto \low E_t(f)$ on a required time interval. Typically it is of the form $[0, T]$, where $0$ denotes the initial time of the process observation. The value of $\low E_t(h)$ depends on the initial distribution, represented by an initial lower expectation $\low E_0$, and the \emph{transition law}, which is described in terms of imprecise transition rates, as described in the following section. 
	
	\subsection{Imprecise transition rate matrices}\label{ss-itrm}
	A continuous-time Markov process switches between states in $\states$ randomly according to some \emph{transition rates}, which are described using \emph{$Q$-matrices}, also named \emph{transition rate matrices}. Each element $Q_{kl}$, for $k\neq l$, of a transition rate matrix denotes the rate at which a process in state $k$ moves to state $l$. Its value is non-negative. The diagonal elements $Q_{kk}$ are negative and denote the rate of leaving $k$. It follows that $Q_{kk}=-\sum_{l\neq k}Q_{kl}$, which implies that the sum of all rows of a $Q$-matrix equals 0. 
	
	If the process is governed by a precise $Q$-matrix, i.e. with constant transition rates, the expectations corresponding to $X_t$ are calculated as $E_t (f) = E_0 (e^{tQ}f)$ for a gamble $f\in \RR^m$ (see e.g. \cite{de2017limit, skulj:15AMC}). This formula, however, does not allow direct generalization to the imprecise case. Therefore, we rather turn to its differential version, where another modification is carried out. That is, we shift the focus from the calculation of the transformed probability distributions to calculating the transformed gambles as functions of time. This becomes more apparent after denoting the transition operator $T_t = e^{tQ}$ acting on the set of gambles. We have that $E_t (f) = E_0 (T_t f)$. The calculation of $E_t(f)$ consequently translates to the calculation of the expectation of $T_tf$ with respect to the initial model $E_0$. The transition operator $T_t$ satisfies the Kolmogorov backward equation 
	\begin{equation}\label{eq-kbe}
		\frac{\dd}{\dd t}T_t f = QT_t f
	\end{equation}
	for every gamble $f$. This differential equation does allow involving imprecision via replacing a precise transition rate matrix $Q$ with an imprecise generalization as introduced below. 
	
	The imprecision in transition rates is modelled by replacing precisely given transition rate matrices with sets of those, called \emph{imprecise transition rate matrices} or \emph{imprecise $Q$-matrices}. These sets are assumed to contain the factual transitions governing the dynamics of the system at any time $t$, and are typically denoted by $\qset$. Thus at every time we merely assume that transition rates belong to the set $\qset$, while in the course of time they may arbitrarily vary within it. We additionally require the imprecise $Q$-matrices to be closed, convex and bounded, i.e. there exists a constant $M$ such that $| Q_{kl}| \leqslant M$ for every $Q\in\qset$ and $k, l \in \{ 1, \ldots , m\}$. 
	
	Let $\mathcal Q$ be an imprecise $Q$-matrix. Fixing a row index $k$, let $\mathcal Q_k\colon \RR^m \to \RR$ be the set of functionals defined by $Q_k(f) = [Qf]_k$ for every $Q\in \mathcal Q$ and $f\in \allgambles$. We say that $\mathcal Q$ has \textit{separately specified rows} if for every collection of $Q_k\in \mathcal Q_k$, for $k\in \{ 1, \ldots, m\}$, there exists a matrix $Q \in \mathcal Q$ whose $k$-th row is $Q_k$: $[Q(f)]_k = Q_k(f)$. In other words, a set of matrices $\mathcal Q$ has separately specified rows if $\mathcal Q = \times_{k\in \states} \mathcal Q_k$. From now on, the separately specified rows property will be added to the list of standard requirements for an imprecise $Q$-matrix.
	
	For an imprecise $Q$-matrix, the corresponding \emph{lower transition operator} is defined by
	\begin{equation}\label{eq-lto}
		\low Q f := \min_{Q\in \mathcal Q} Q f,
	\end{equation}
	where the $\min$ is meant componentwise. However, the separately specified rows property ensures that for every $f\in \allgambles$, some $Q_f\in\mathcal Q$ exists such that $Q_f f = \low Q f$. Thus, the above componentwise minimum is actually attained by some product $Q_f f$. 
	
	Below we list some fundamental properties of lower transition rate operators. Let gambles $f, g\in \allgambles$, the constant gamble $\mu 1_\states, \lambda \geqslant 0$ and a row index $k$ be given. The following properties hold:
	\begin{enumerate}[(i)]
		\item $\low Q (\mu 1_\states) = \mathbf 0$;
		\item $[\low Q 1_{\{l\}}]_k \geqslant 0$ for all $l\in \states$ such that $l\neq k$;
		\item $\low Q (f+g) \geqslant \low Q f + \low Q g$;
		\item $\low Q (\lambda f) = \lambda \low Q f$.
	\end{enumerate}
	In the above relations and elsewhere, the inequality relations such as $f\leqslant g$ between vectors are meant to denote $f(k)\leqslant g(k)$ for every $k\in \states$. 
	
	The converse the above is also true, i.e. that for every operator $\low Q$ satisfying the above properties (i)--(iv), an imprecise $Q$-matrix $\qset$ exists such that
	\begin{equation}
		\mathcal Q = \{ Q\colon Q \mu 1_\states = \mathbf 0, Qf \geqslant \low Q f \text{ for every } f\in \allgambles \}.
	\end{equation}
	The proof of the above one-to-one correspondence can be found in \cite{krak2017imprecise}.

	\subsection{Distributions at time $t$}	
	Consider again the Kolmogorov backward equation \eqref{eq-kbe} and its relation with the expectation functional $E_t = E_0T_t$, which uniquely characterizes the distribution at time $t$ for the precise case. Transferring the equation to the imprecise case amounts to replacing $E_t$ with its imprecise version $\low E_t$, which is obtained as the product of the imprecise versions of $E_0$ and $T_t$. The imprecise initial distribution is modelled by the lower expectation $\low E_0$. The transition law in the imprecise case will be modelled by the \emph{lower transition operator} $\low T_t$. Lower (and upper) transition operators and their properties are in fact well-known from the discrete time theory, which has been successfully transferred to the imprecise case a while ago (see e.g. \cite{decooman-2008-a, skulj:09IJAR}). 
	
	The imprecise distribution of $X_t$ represented by the lower expectation functional $\low E_t$ now satisfies the following relation \cite{skulj:15AMC}: 
	\begin{equation}
		\low E_t (f) = \low E_0 (\low T_t f),
	\end{equation}
	for every gamble $f$. The lower transition operator $\low T_t$ is obtained as the unique non-linear operator satisfying 
	\begin{equation}\label{eq-lto}
		\frac{\dd}{\dd t}\,\low T_t f = \low Q\, \low T_t f 
	\end{equation}
	and the initial condition $T_0f = f$ for every gamble $f$. Actually, De Bock~\cite{de2017limit} showed the above equation holds even without reference to a specific gamble $f$. Yet, finding a specific lower expectation is merely possible for a given $f$ in which case both interpretations of the equation coincide.
	
	To calculate $\low E_t(f)$ for a specific vector $f$, the lower operator $\low T_t$ does not need to be completely specified. instead only the vector function $f_t := \low T_t f$ needs to be evaluated. By \eqref{eq-lto}, it follows that 
	\begin{equation}\label{eq-BSDE}
		\frac{\dd}{\dd t} f_t = \low Q f_t, 
	\end{equation}
	with the initial condition $f_0 = f$. 
	It was shown in \cite{skulj:15AMC} that this equation has a unique solution for a lower transition rate operator satisfying (i)--(iv) from section \ref{ss-itrm}. 
	
	Unfortunately, no analytical formula similar to the matrix exponential solving the precise version has been discovered in general imprecise case. (For the case of $m=2$, an explicit formula has been found in \cite{erreygers2017imprecise}). This leaves us depending on more or less efficient numerical methods. The goal of the remainder of the paper is proposing an efficient numerical method based on the theory of normal cones. 
	
	\section{Numerical methods for finding lower expectations}\label{s-nmfle}
	
	\subsection{Lower expectation and transition operators as linear programming problems}
	The methods for finding lower expectations of the random variables $X_t$ are based on linear programming methods. As explained in the previous section, coherent lower (or upper) previsions are often presented in the form of a finite number of assessments, which can be turned into constraints of linear programming problems. Something similar can be said for imprecise transition rates, which as convex sets can also be generated by imposing a finite number of linear constraints. The corresponding objective function is usually deduced from the minimizing gamble. 
	
	Specifically, consider Equation~\eqref{eq-BSDE}. The calculation of the lower transition rate $\low Qf_t$ for a given $f_t$ is an optimization problem, where the minimum 
	\begin{equation}
		\min_{Q_k\in \mathcal Q_k} Q_k(f) 
	\end{equation}
	has to be obtained for every component $k\in \{ 1, \ldots, m\}$. If the set $\mathcal Q_k$ is represented by a finite number of constraints, the above optimization problem can be solved by linear programming techniques. Once the solution $h_k = \min_{Q_k\in \mathcal Q_k} Q_k(f)$ is obtained for every $k$, the solutions are combined into the solution vector $h$, whose components are $h_k$, and the minimizing matrix $Q$, whose rows are exactly the minimizing solutions $Q_k$. 
	
	\subsection{Finitely generated sets of transition rate matrices}
	It is common in the theory of imprecise probabilities that judgements are given for a certain class of gambles, such as indicator gambles $1_A$. We can thus, for instance, say that the transition rate from a state $k$ to a set $A$ is at least $2$. Then we write $Q_k(1_A) \geqslant 2$. Typically, we might have a finite set $\gambleset$ of gambles together with a set of judgements $Q_k(f)\geqslant \low Q_k(f)$, where $\low Q_k(f)$ are prescribed lower transition rates. To make the linear programming approach applicable, judgements about transition rates have to be supplemented by the general conditions for $Q$-matrices. 
	
	Thus, we assume that the judgements about the transition rates are given in the form $\low Q_k (f)$ for every $f\in \gambleset$. It would be possible of course that the sets $\gambleset$ would depend on $k$ as well, but for convenience we will stick with the shared set of gambles. The methods for the more general case, however, would be directly derived from the methods presented here. 
	
	Now an imprecise $Q$-matrix $\qset$ can be formed as 
	\begin{multline}\label{eq-imprecise-matrix}
		\mathcal Q = \{ Q\colon Q_k 1_\states = 0, Q_k(f) \geqslant \low Q_k(f)~ \forall k\in \states~\forall f\in \gambleset, \\ Q_k (1_{\{l\}}) \geqslant 0 ~\forall l\neq k ~\forall k \in \states \}.
	\end{multline}
	Our general assumption is that the imprecise $Q$-matrix has separately specified rows. This property is clearly satisfied if the constraints on $Q_k(f)$ are independent from those on $Q_l(f')$ for $k\neq l$. Imposing a constraint, for instance, in the form $Q_k (f) = Q_l (f)$ would restrict the choice of matrix rows in the set $\mathcal Q_l$ once $Q_k$ is selected.
	
	\subsection{Computational approaches to estimating lower expectation functionals} 
	The most common computation involving imprecise continuous-time Markov chains is solving of the Equation~\eqref{eq-BSDE} for a given gamble $f$ on a finite time interval $[0, T]$. The calculation of $\low Qf_t$ is typically implemented as a linear programming problem. In principle it would have to be solved for every single time point of an interval, and this is clearly impossible. Hence, the exact solution is in most cases unattainable, whence we have to satisfy with approximations. 
	
	Most of the computational approaches to finding approximate solutions proposed in literature apply some kind of discretization of the interval $[0, T]$. This means constructing a sequence $0=t_0 < t_1 < \dots < t_n = T$. By the semigroup property of the lower transition operators, we then have that $\low T_T = \prod_{i=1}^{n} \low T_{t_i - t_{i-1}}$. The idea is now to take the differences $\delta t = t_i-t_{i-1}$ sufficiently small, so that approximations of the form $\hat{\low T}_{\delta t} = (I+\delta t Q)$ or $\hat{\low T}_{\delta t} = e^{\delta t Q}$, for some matrix $Q$, minimizing $\low Qf_{t_{i-1}}$, are accurate enough even when the approximation errors compound. It has been shown in \cite{skulj:15AMC,krak2017imprecise,erreygers2017imprecise} that it is possible, with appropriately fine grids, to achieve arbitrarily accurate approximations.   	
	The approximate solution $\hat f_T$ of $\low T_Tf$ is then obtained by initially setting $f_0 = f$ and then sequentially calculating the approximations $\hat f_{t_i} = \hat{\low T}_{t_i - t_{i-1}} \hat f_{t_{i-1}}$, resulting ultimately in $\hat f_T = \hat f_{t_n}$. The present methods differ in the way the step sizes $t_i-t_{i-1}$ are determined and how the approximate transition operators $\low T_{t_i - t_{i-1}}$ are obtained. 
	
	Our goal is to achieve a progress in the applicability of the approach presented in \cite{skulj:15AMC}, called \emph{the approximation with adaptive grid method}. To explain the underlying idea, note that the optimization problems for finding the minima $\low Q f_t$ for different $t$ are all the same as far as constraints are considered, and they merely differ in the objective functions, which correspond to $f_t$, which is a Lipschitz continuous function of $t$ (cf. Proposition~7 in \cite{skulj:15AMC}). Therefore, it is legitimate to expect that the matrices $Q$, minimizing expression $Q f_t$, would lie in a close neighbourhood, or even be the same, for proximate values of $t$. This idea is unique to our approach, as the majority of other methods in existence do not attempt to make use of the continuity of solutions $f_t$. 
	
	By our method, the intervals $t_i-t_{i-1}$ are chosen in the way that the corresponding transition operators $\low T_{t_i - t_{i-1}}$ can be approximated by $e^{(t_i - t_{i-1})Q}$, where $Q$ is a transition rate matrix. Very often, this choice even produces the exact solution on a suitable interval, i.e. no error additional to the initial error of $\hat f_{t_{i-1}}$ is produced. Moreover, utilizing this method, the intervals $t_i-t_{i-1}$ are typically allowed to be considerably wider than with using the alternative techniques. 
	
	Another adaptive grid method has also been proposed in \cite{erreygers2017imprecise}, which uses intervals of varying lengths, yet the choice of the length is not based on the same assumption. Instead, they allow intervals to become wider based on the convergence of solutions in a suitable norm. 
	
	In the previous paper \cite{skulj:15AMC}, the implementation of the adaptive grid method was introduced, yet it lacks a fast implementation. In this paper we improve the approach presented there in two directions. First we provide a much more efficient way of analyzing the maximal possible error, which effectively answers whether the approach is feasible on the given interval. The second improvement is the approximate version of the method, which can always be applied if only the intervals are made small enough. The error of the approximate version is in the worst case merely comparable with the ordinary grid methods, while in most cases being significantly smaller. Both improvements arise from the new foundations based on the theory of normal cones of convex sets.

	\section{Normal cones of imprecise Q-operators}\label{s-nciq}
	A closed and convex set of transition matrices generated in the form of \eqref{eq-imprecise-matrix} is a convex polyhedron if the set of constraints is finite and it is non-empty and bounded. Moreover, if it additionally satisfies the separately specified rows property, it can be represented as a product of the row polyhedra $\mathcal Q = \times_{k=1}^m \mathcal Q_k$. 
	
	\subsection{Normal cones of convex sets}	
	We start our introduction to normal cones with general vector spaces.
	Let $V$ be a finite dimensional vector space equipped with the standard inner product. A \emph{convex polyhedron} in $V$ is a bounded convex set $\mathcal C$ with finitely many extreme points. Equivalently, a convex polyhedron can be represented as an intersection of a finite number of half spaces of the form $\{ x\in V \colon x f \geqslant b_f \}$, where $f\in V$ is a given vector, $b_f$ is a constant and $x f$ denotes the standard inner product of $x$ and $f$. Thus, we can write 
	\begin{equation}\label{eq-convex-polyhedron}
		\mathcal C = \{ x\in V \colon x f \geqslant b_f \text{ for all } f \in \gambleset \},
	\end{equation}
	where $\gambleset$ is a given finite collection of vectors. Some of the inequalities $x f\geqslant b_f$ may in fact be equalities, such as in the representation of the imprecise $Q$-matrix, where $q1_\states = 0$ is required. This case, however, can be unified with the general case by replacing an equality condition $x f = b_f$ with two inequalities, $x f \geqslant b_f$ and $x (-f) \geqslant -b_f$. 
	
	Now take some point $x\in \mathcal C$ and define its \emph{normal cone} to be the set 
	\begin{equation}
		N_{\mathcal C} (x) = \{ f\in V \colon x f \leqslant y f \text{ for every } y\in \mathcal C \}.
	\end{equation}
	That is, the normal cone of $x$ is the set of all vectors $f$ for which $x = \arg\min_{y\in\mathcal C}y f$. Most often the minimum of the above expression is recognized as a linear programming problem where $\mathcal C$ is the feasible set. Thus the normal cone of $x$ can be understood as the set of all vectors $f$ such that the objective function $y f$ has an optimal solution in $x$. It is well-known that only points in the boundary minimize objective functions, and therefore only normal cones for those sets are non-empty. Moreover, every objective function is minimized in at least one extreme point. This implies that the union of the normal cones of extreme points is the entire space $V$. 
	
	The following proposition holds (see \cite{gruber:07CDG}, Proposition 14.1).
	\begin{proposition}\label{prop-cone-positive-hull}
		Let $\mathcal C$ be a convex polyhedron represented in the form \eqref{eq-convex-polyhedron} and $x\in \mathcal C$ a boundary point. Let $\gambleset_x = \{ f\in \gambleset\colon x f = b_f \}$. Then 
		\begin{equation}
			N_{\mathcal C} (x) = \mathrm{posi}\, \gambleset_x. 
		\end{equation}
		(The notation $\mathrm{posi}\, \gambleset$ denotes the cone of all non-negative linear combinations of elements in $\gambleset$.)
		
		Moreover, if $x$ is an extreme point of $\mathcal C$, then $\dim N_{\mathcal C} (x) = \dim V$. 
	\end{proposition}
	The final statement of the above proposition implies that for every extreme point $x\in \mathcal C$ the rank of $\gambleset_x$ is $m = \dim V$. Besides, every $y\in N_{\mathcal C} (x)$ is a positive linear combination of of the vectors in $\gambleset_x$. The following proposition additionally holds.
	\begin{proposition}\label{prop-linear-independent-positive-combination}
		Let $h\in N_{\mathcal C} (x)$. Then there exists a linearly independent subset $\gambleset^I_x\subseteq \gambleset_x$ such that $h \in \mathrm{posi}\, \gambleset^I_x$. 
	\end{proposition} 
	\begin{proof}
		Let $\gambleset'\subseteq \gambleset_x$ be a minimal set such that $h \in \mathrm{posi}\, \gambleset'$. To show that $\gambleset'$ is linearly independent, we use the method of contradiction. Hence, suppose that $\gambleset'$ is linearly dependent. Then there exists a non-trivial linear combination $\sum_{f\in \gambleset'} \beta_f f = \mathbf{0}$. Further let $\sum_{f\in \gambleset'} \alpha_f f = h$, where all $\alpha_f > 0$ by the assumption of minimality of $\gambleset'$. As there exists at least one $\beta_f \neq 0$, we can find some constant $c$ such that $\alpha_f + c\beta_f$ is zero for some $f$ and remains positive for the others. We then still have that $\sum_{\gambleset'} (\alpha_f + c\beta_f)f = h$ with at least one coefficient equal 0 and all others positive. Thus $h$ is a positive combination of a set strictly included in $\gambleset'$, which contradicts its minimality. This contradiction now confirms that $\gambleset'$ needs to be linearly independent. 
	\end{proof}
	\begin{corollary}\label{cor-positive-basis}
		Let $h\in N_{\mathcal C} (x)$, where $x$ is an extreme point of $\mathcal C$. Then there exists a basis $\gambleset^I_x\subseteq \gambleset_x$ of $V$, such that $h\in \mathrm{posi}\, \gambleset^I_x$. 
	\end{corollary}
	\begin{proof}
		By Proposition~\ref{prop-cone-positive-hull}, the rank of $\gambleset_x$ equals the dimension of $V$. Moreover, by Proposition~\ref{prop-linear-independent-positive-combination}, $h$ is a positive linear combination of an independent subset of $\gambleset_x$. Now this subset can be completed with elements of $\gambleset_x$ to a basis of $V$, and the added vectors can be also be added to the positive linear combination with zero coefficients, thus forming a positive linear combination of the basis. 
	\end{proof}
	The above corollary is essential for our method which is based on representing gambles $f$ as non-negative linear combinations of bases consisting of elements of $\gambleset$ that lie in the same normal cone as $f$. 
	
	\subsection{Normal cones of imprecise transition rate matrices}
	In the case of imprecise Q-matrices denoted generically by $\qset$, we assumed that it has separately specified rows which implies that it is of the form $\mathcal Q=\times_{k\in \states} \mathcal Q_k$, where each $\mathcal Q_k$ is a convex polyhedron of vectors $q_k$, represented by the constraints 
	\begin{align}
		q_k f & \geqslant \low Q_k(f) & \text{for every $f\in \gambleset$}\label{eq-const1};\\ 
		q_k 1_{\{l\}} & \geqslant 0 & \text{for every $l \neq k$}\label{eq-const2}; \\ 
		q_k 1_\states &= 0. \label{eq-const3}
	\end{align} 
	\begin{remark}
		Note that we have now switched the notation of matrix rows, previously denoted by $Q_k$, to $q_k$. This is because we now view the rows as row vectors instead of parts of particular matrices. They do still form transition matrices together with other rows, but the focus is now more on the rows as elements of the row set $\qset_k$. When the rows correspond to explicitly mentioned matrices, we will still use the notation of the form $Q_k$. 
	\end{remark}

	\begin{remark}
		It might seem that constraints \eqref{eq-const1} and \eqref{eq-const2} are not general enough because of the $\geqslant$ form. However, it is readily verified that constraints of the form of inequalities $\leqslant$ or with an equality sign can be easily represented either by changing the sign or forming two reversed inequalities instead of an equality. 
	\end{remark}
	\begin{remark}
		Constraints \eqref{eq-const2} are of the same form as \eqref{eq-const1}, and could be even implied by the latter. Therefore, we adopt the convention that the gambles of the form $1_{\{j\}}$ are always assumed to be contained in $\gambleset$, together with the corresponding constraints and are removed if they are already implied by the remaining constraints. The primary reason for this is a simplified notation. Yet, the constraint \eqref{eq-const3} we choose to separate from the inequality constraints and therefore also not consider $1_\states$ as an element of $\gambleset$. 
	\end{remark} 
	Take a row set $\mathcal Q_k$, which is a convex set of vectors: 
	\[ \mathcal Q_k = \{ q\in \RR^m \colon q 1_\states = 0, q f \geqslant \low Q_k (f) ~\forall f \in \gambleset \} . \]
	For every element $q\in \mathcal Q_k$, the corresponding normal cone is the set of vectors 
	\[ N_{\mathcal Q_k} (q) = \{ f\in \allgambles\colon q f \leqslant p f~\forall p\in \mathcal Q_k \}. \]
	(See e.g. \cite{gruber:07CDG}.)
	Vector $q$ can be considered as a $k$-th row of a matrix $Q\in \mathcal Q$, and its normal cone is the set of all vectors $f\in \allgambles$ for which $q = \arg \min_{q'\in \mathcal Q_k}q' f$. 
	
	To simplify the notation, we will now assume the gambles in $\gambleset$ are enumerated by some indices $i\in I$, where $I$ is an index set. Thus $\gambleset = \{ f_i \colon i\in I\}$. 
	By Proposition~\ref{prop-cone-positive-hull}, every element $f$ of the normal cone $N_{\mathcal Q_k} (q)$ can be represented as a linear combination of elements in $\gambleset$ that are contained in the cone: 
	\begin{equation}\label{eq-vector-decomposition}
		f = \sum_{i\in I_q} \alpha_i f_i + \alpha_0 1_\states, 
	\end{equation}
	where $I_q = \{i\in I\colon q f_i = \low Q(f_i)\}$; $\alpha_i \geqslant 0$ for all $i\in I_q$ and $\alpha_0$ is an arbitrary real constant. Here we used the fact that the constraint $q1_\states = 0$, can equivalently be stated as a combination of two distinct constraints, $q1_\states \geqslant 0$ and $q(-1_\states) \geqslant 0$, and therefore, depending on the sign of $\alpha_0$, either $1_\states$ or $-1_\states$ appears in the above linear combination with a positive coefficient. We will call the subset $\gambleset_{I_q} = \{ f_i \colon i\in I_q \}$ the \emph{basis of the cone} $N_{\mathcal Q_k} (q)$. 
	

	\section{Norms of $Q$-matrices}\label{s-nqm}
	In our analysis we will use vector and matrix norms. For vectors $f$ we will use the maximum norm
	\begin{equation}\label{eq-l1}
		\Vert f \Vert = \max_{i\in \states} |f_i|,
	\end{equation}
	and the corresponding operator norm for matrices
	\begin{equation}\label{eq-l1-matrix}
		\Vert Q \Vert = \max_{1\leqslant k\leqslant m} \sum_{l=1}^m |q_{kl}|.
	\end{equation}
	For every stochastic matrix $P$, we therefore have that $\Vert P \Vert = 1$, which implies that $\Vert e^Q \Vert = 1$ for every $Q$-matrix. In general, $Q$ matrices may have different norms, though.
	
	For a bounded closed set of vectors $\gambleset$ we will define 
	\begin{equation}
		\Vert \gambleset \Vert = \max_{f\in\gambleset} \Vert f \Vert 
	\end{equation}
	and for a bounded closed set of matrices $\qset$
	\begin{equation}
		\Vert \qset \Vert = \max_{Q\in\qset} \Vert Q \Vert .
	\end{equation}
	It has been shown in \cite{erreygers2017imprecise} that for an imprecise $Q$-matrix 
	\[ \| \qset \| = 2\max\left\{ \left|[\low Q 1_{\{ k \}}]_k\right|\colon k\in \states \right\} \] 
	holds, where $\low Q$ is the corresponding lower transition operator.

	The distance between two vectors $f$ and $g$ is defined as $d(f, g) = \Vert f-g \Vert$, and the maximal distance between two elements of a set of vectors $\gambleset$ will be called the \emph{diameter} of the set and denoted with $\delta(\gambleset) = \max_{f, g\in\gambleset} d(f, g)$. Additionally, we define the distance between two matrices as $d(Q, R) = \Vert Q-R \Vert$, while the diameter of an imprecise $Q$-matrix $\qset$ we pronounce as the \emph{imprecision} of $\qset$, denoted by $\iota(\qset) = \max_{Q, R\in \qset} d(Q, R)$. The degree of imprecision has been previously defined in \cite{vskulj2017perturbation} in the $L_1$ metric for the case of imprecise discrete time Markov chains. 
	
	The following proposition is immediate.
	\begin{proposition}\label{prop-bound-iota}
		Let $\qset$ be an imprecise $Q$-matrix. Then $\iota(\qset) \leqslant 2\Vert \qset \Vert $. 
	\end{proposition}
	\begin{proposition}
		Let $\qset$ be an imprecise $Q$-matrix and $\low Q$ its associated lower transition operator. Then we have that $\| \low Q f - \low Qf' \| \leqslant \| \qset \| \| f - f'\|$ for every pair of gambles $f, f'\in\allgambles$. 
	\end{proposition}
	\begin{proof}
		By definition we have that 
		\begin{equation*}
			\| \low Qf - \low Qf' \| = \max_{i\in \states} | \low Q_i (f) - \low Q_i (f') |. 
		\end{equation*}
		Now for every $i\in \states$, the following inequality follows from basic properties of lower envelope operators
		\begin{equation*}
			\low Q_i (f') + \low Q_i (f-f') \le \low Q_i(f) \le \low Q_i(f') + \up Q_i(f-f'),
		\end{equation*}
		implying further that 
		\begin{equation*}
			\low Q_i (f-f') \le \low Q_i(f) - \low Q_i (f')  \le \up Q_i(f-f'),
		\end{equation*}
		and hence 
		\begin{equation*}
			|\low Q_i(f) - \low Q_i (f')| \le	\max\{|\low Q_i (f-f')|, |\up Q_i(f-f')|\}.
		\end{equation*}
		Moreover, it follows by the definition of the lower and upper envelope operators that
		\begin{equation*}
			\low Q_i (f-f')\le Q_i (f-f') \le \up Q_i(f-f'),
		\end{equation*}
		for every $Q_i\in \qset_i$ and 
		\begin{equation*}
			\max_{Q_i\in \qset_i} |Q_i(f-f')| = \max\{|\low Q_i (f-f')|, |\up Q_i(f-f')|\}.
		\end{equation*}
		By separately specified rows property and compactness of $\qset$, there actually exists a matrix $\tilde Q\in \qset$ such that $|\tilde Q_i(f-f')| = \max\{|\low Q_i (f-f')|, |\up Q_i(f-f')|\}$ for every $i\in \states$. Summarizing the above equations gives:
		\begin{align*}
			\| \low Qf - \low Qf' \| = & \max_{i\in \states}|\low Q_i(f)-\low Q_i(f')| \\
			\le & \max_{i\in \states} \max\{|\low Q_i (f-f')|, |\up Q_i(f-f')|\} \\
			= & \| \tilde Q(f-f') \|=  \max_{Q\in \qset} \| Q(f-f') \| \\
			\le &  \max_{Q\in \qset} \| Q \| \| (f-f') \|   = \| \qset \| \| f - f'\|. 
		\end{align*} 		
	\end{proof}
	
	In the literature, the variational seminorm 
	\begin{equation*}
		\| f \|_v = \max f - \min f 
	\end{equation*}
	is also often used, and proves especially useful in the context of stochastic processes. In \cite{erreygers2017imprecise} the quantity $\| f\|_c = \frac{1}{2}\| f \|_v$ is also used. The reason to turn from norms to the seminorm is in the simple fact that $\| f\|_v = 0$ implies that $f$ is constant and further that $Qf = 0$ for every $Q$-matrix $Q$ and $Tf = f$ for every transition operator $T$. Moreover, $\| Tf \|_v \leqslant \| f\|_v$ holds for every $f\in \allgambles$. The inequality $\| f\|_c \leqslant \| f \|$ is also immediate. 
	\begin{proposition}\label{prop-q-matrix-variational}
		Let $Q$ be a $Q$-matrix and $f\in \allgambles$ a gamble. Then $\| Q f \| \leqslant \| Q \| \| f \|_c$. 
	\end{proposition}
	\begin{proof}
		Let $f_M = \frac{\max f + \min f}{2}$ and $f_V = f-f_M$. Clearly $\| f_V \| = \frac{\max f-\min f}{2} = \| f \|_c$ and, as $f_M$ is a constant, $\| Qf \| = \| Q(f_M + f_V) \| = \| Qf_V \| \leqslant \| Q \| \| f_V \| = \| Q \| \| f \|_c$. 
	\end{proof}
	\begin{corollary}
		Let $\qset$ be an imprecise $Q$-matrix and $\low Q$ the associated lower transition operator. Then for all $f\in \allgambles, ~\|\low Q f \| \leqslant \| \qset \| \| f \|_c$. 
	\end{corollary}
	
	\begin{proposition}
		Let $\qset$ be an imprecise $Q$-matrix and $\low Q$ the associated lower transition operator. Further take some extremal matrix $Q\in \qset$ and $h$ a vector such that $h = h_n + h_e$, where $h_n\in N_{\mathcal Q} (Q)$. Then $\Vert Qh - \low Qh \Vert \leqslant \iota(\qset)\Vert h_e \Vert_c \leqslant 2\| \qset \|\| h_e \|_c \leqslant 2\| \qset \| \| h_e \|$. 
	\end{proposition}
	\begin{proof}
		Take some $h = h_n + h_e$. Since $h_n\in N_\qset(Q)$, it follows that $Qh_n = \low Qh_n$. 
		Using superadditivity of $\low Q$ we obtain 
		\begin{multline*}
			\| Qh - \low Qh \| \leqslant \| Qh_n + Q h_e - \low Qh_n -\low Q h_e \| \\ 
			\leqslant \| Q h_e - \low Qh_e \| \leqslant \iota(\qset)\| h_e \|_c \leqslant 2\| \qset \| \| h_e \|_c \leqslant 2\| \qset \| \| h_e \|,
		\end{multline*}
		where the penultimate inequality follows from Proposition~\ref{prop-bound-iota}. 
	\end{proof}
	
	\section{Numerical methods for CTIMC bounds calculation}\label{s-nmbc}
	In this section we discuss methods for calculation of the solutions of the differential equation \eqref{eq-BSDE}. Let $h_t$ be a solution of this equation 
	with the initial value $h_0$. The initial value may be an approximation at a previous stage or interval. 	
	It has been shown in \cite{skulj:15AMC} (Proposition 7) that the solution $h_t$ is Lipschitz continuous. More precisely, the following estimate holds
	\begin{equation}\label{eq-h-lipschitz}
		\| h_{t+\Delta t} - h_t \| \leqslant \Delta t \| \qset \| \| h_0\| e^{\Delta t \| \qset \|} = \Delta t \| \qset \| \| h_0 \| + o(\Delta t).   
	\end{equation}
	\subsection{Matrix exponential method}	
	Assume that the initial vector $h_0 = h$ is given and let $Q$ be an extreme $Q$-matrix such that $Qh = \low Qh$. By definition, the initial vector belongs to the collection of normal cones $N_{\qset_k}(Q_k)$ for every $k\in \{ 1, \ldots, m\}$. Thus, for each index $k$, we have an index set $J\subseteq I$, such that $\gambleset_{J}$ forms the basis of $N_{\qset_k}(Q_k)$. Moreover, by Corollary~\ref{cor-positive-basis}, a basis $\gambleset_{\tilde J}$ of $\RR^m$ exists, so that $h$ is a non-negative linear combination of elements of $\gambleset_{\tilde J}$. In our case, the basis contains either $1_\states$ or $-1_\states$, which are excluded from the set of gambles indexed by $I$. Let $I_{k, h}$ denote the index set which together with $1_\states$ or $-1_\states$ forms the required basis corresponding to the $k$-th row. Then we can write: 
	\begin{equation}\label{eq-nc-decomposition}
		h = \sum_{i\in I_{h, k}} \alpha_{ki}f_i + \alpha_{k0} 1_\states,
	\end{equation} 
	where $\alpha_{ki} \geqslant 0$ for every $i \in I_{h, k}, \alpha_{k0} \in \RR$. By these assumptions, the solution $h_t$ of equation \eqref{eq-BSDE} can be written as a linear combination of the form \eqref{eq-nc-decomposition} for every $t\geqslant 0$, yet not necessarily with non-negative coefficients $\alpha_{ki}$ for $t\gg 0$.
	\begin{remark}
		In the case described above where $\gambleset\cap N_{\qset_k}(Q_k)$ is not linearly independent, instead of the entire normal cone we only consider its part that contains the gamble $h$ and is positively spanned by the linearly independent subset. In principle such a set may only represent a fraction of the normal cone. In order to avoid repeating this fact, we will from now on slightly abuse terminology to name a cone spanned by a linearly independent set a normal cone. Yet, apart from the definition, this fact does not have any other negative impact, as these subsets of the normal cones are cones as well and they may likely become normal cones if only the constraints are slightly changed. 
	\end{remark}
	
	In the general case, the vector $h = h_0$ would belong to the interior of a normal cone, whence the coefficients $\alpha_{ki}$ are all strictly positive. For a small enough time $T>0$, the values of $h_t$ may still belong to the same normal cone, whence they would satisfy $\low Q h_t = Q h_t$, for every $t\in [0, T]$. In that case, the exact solution $h_T$ can be found explicitly as $h_T = e^{TQ}h_0$. Quite surprisingly, it has been shown in \cite{skulj:15AMC} that checking whether the above condition holds is possible by merely considering the solution at the end-point $T$. More precisely, we need to consider the partial sums corresponding to the solution. An implementation of this exact method was also proposed in the same paper, yet here we improve significantly on its efficiency by making use of the normal cones. 
	
	To employ the exact method efficiently, it is necessary to aptly implement the following steps:
	\begin{itemize}
		\item finding the time interval $T$ where the method is applicable with $T$ as large as possible,
		\item verify whether the method is applicable on a given interval $[0, T]$ with acceptable maximal possible error. 
	\end{itemize}
	The second step suggests we might have an interval where the solutions $h_t$ do not lie exactly in the required normal cone, but sufficiently close to it, so that the error remains within acceptable bounds.

	\subsection{Finding a linearly independent positive linear combination}\label{ss-fliplc}
	Propositions~\ref{prop-cone-positive-hull} and \ref{prop-linear-independent-positive-combination} ensure that the set $\gambleset_{I_{h, k}}$ satisfying \eqref{eq-nc-decomposition} can be chosen so that together with $1_\states$ it forms a basis of $\allgambles$. Finding this set however is not a trivial task. This is because it requires finding a non-negative solution to a system of linear equations. We are therefore looking for a solution $(\alpha_{ki})_{i\in I_{h, k}\cup \{ 0 \}}$ of equation \eqref{eq-nc-decomposition} such that $\alpha_{ki}\geqslant 0$ for every $i\in I_{h, k}$, while $\alpha_{k0}$ can be arbitrary. The exception of $\alpha_{k0}$ can be handled by adding $\alpha_{k0+}$ and $\alpha_{k0-}$ corresponding to the vectors $1_\states$ and $-1_\states$ respectively, which can clearly be required both non-negative.

	The above problem is known as the auxiliary problem in the two phase simplex method, which can be stated as a linear programming problem for minimizing the objective function $\sum_{i\in I_{h, k}\cup \{ 0 \}}\alpha_{ki}$ subject to $A\alpha_k = b$ and $\alpha_k \geqslant 0$. Knowing the solution exists, this is a routine linear programming task. 
	
	Once a solution $\alpha_k\geqslant 0$ has been found, we proceed by eliminating the vectors $f_i$ from $\gambleset_{I_{h, k}}$ in the way that can be directly deduced from the proof of Proposition~\ref{prop-linear-independent-positive-combination}, until they form a linearly independent set. Therefore, if the vectors are not linearly independent, a linear combination $\sum_{i \in I_{h, k}} \beta_{ki}f_i + \beta_{k0} 1_\states = 0$ exists. Further, take a collection of coefficients so that $\sum_{i \in I_{h, k}} \alpha_{ki}f_i + \alpha_{k0} 1_\states = h$. We now take a suitable constant $c$ so that $(\alpha_{ki}+c\beta_{ki}) \geqslant 0$ for all $i\in I_{h, k}$ and $\alpha_{kj}+c\beta_{kj} = 0$ for some $j\in I_{h, k}$ (note that at least one $\beta_{ki}\neq 0$). Thus, we have obtained a new solution to equation \eqref{prop-linear-independent-positive-combination}, with $f_j$ omitted. This procedure completes with a linearly independent set $\gambleset_{I_{h, k}}$.

	The obtained set $\gambleset_{I_{h, k}}$, however, may not form a basis of $\allgambles$ in which case we complete it to a basis using the remaining vectors from the normal cone. This is possible because by Proposition~\ref{prop-cone-positive-hull} the set has full rank. Obviously, the added gambles appear in the linear combination with zero coefficients.  
		
	\subsection{Checking applicability of the matrix exponential method}\label{ss-camem}
	The procedure of checking the applicability of the exact method to an interval $[0, T]$ is based on the following results proposed in \cite{skulj:15AMC}. 
	\begin{lemma}
		Let 
		\begin{equation}\label{eq-inf-ser}
		\sigma(t) = \sum_{s = 0}^\infty a_s t^s
		\end{equation}
		be a power series that converges in an interval $[0, T]$ and denote its partial sums with
		\begin{equation}\label{eq-part-poly}
		p_r(t) = \sum_{s = 0}^r a_s t^s. 
		\end{equation}
		Then, for every $t \in [0, T]$, we have that 
		\[ p_r(t) = \sum_{s = 0}^r \lambda_s p_s(T) \]
		for some non-negative coefficients $\lambda_s$ for which $\sum_{s =0}^r \lambda_s = 1$. 
	\end{lemma}
	\begin{remark}
		Put differently, the above lemma says that $p_r(t)$ is a convex combination of $p_s(T)$ for $s\in \{ 0, \ldots, r\}$. 
	\end{remark}
	The following corollaries follow immediately. 
	\begin{corollary}\label{cor-matrix-series-cone}
		Let $Q$ be an arbitrary square matrix of order $m$, $h\in \allgambles$ and $\sigma$ a function defined by an infinite power series as in \eqref{eq-inf-ser}. Further let $p_r(tQ)$ be the partial sums \eqref{eq-part-poly}. Then, for every $t \in [0, T]$, some non-negative coefficients $\lambda_s$ satisfying $\sum_{s =0}^r \lambda_s = 1$ exist such that  
		\[ p_r(tQ)h = \sum_{s = 0}^r \lambda_s p_s(TQ)h. \]
	\end{corollary}
	\begin{corollary}\label{cor-single-cone}
	Let $Q$ be an arbitrary square matrix of order $m$, $h\in \allgambles$ and $C\subseteq \allgambles$ a convex set, such that $h\in C$. Further let $p_r(tQ)$ be the partial sums \eqref{eq-part-poly}. If $p_s(TQ)h\in C$ for some $T>0$ and every $s\in \{ 0, \ldots, r\}$, then $p_s(tQ)h\in C$ for every $t \in [0, T]$ and every $s\in \{ 0, \ldots, r\}$.  
		
		In particular, if the above conditions hold for every $r\in \NN$ then $\sigma(tQ)h\in C$ for every $t \in [0, T]$. 
	\end{corollary}
	The above corollary holds for every function $\sigma(tQ)$ with convergent Taylor series on the interval $[0, T]$, however, in this paper the case $\sigma(tQ) = e^{tQ}$ will only be considered. 
	Note also that the converse of the above corollary, and especially its last statement does not hold. Namely, it is quite possible that $\sigma(tQ) = e^{TQ}h\in C$, while $p_r(TQ)h\not \in C$ for some $r$, and in this case $e^{tQ}h\in C$ cannot be guaranteed for all $0<t<T$. 
	
	An approximate version of the above results holds as well. 
	\begin{theorem}\label{thm-normal-decomposition}
		Assume the notation of Corollary~\ref{cor-matrix-series-cone} with $\sigma(tQ) = e^{tQ}$. Suppose that $\varepsilon > 0$ and $T>0$ exist such that	for every $s\in \{ 0, \ldots, r\}$ we can write
		 \[ p_s(TQ)h = h^{C, s}_T + h_T^{E, s}, \] 
		 where $h, h^{C, s}_T\in C$ and $\| h_T^{E, s} \|_c \leqslant \varepsilon$. Then for every $t \in [0, T]$ and $s\in \{ 0, \ldots, r\}$ we have 
		 \[ p_s(tQ)h = h^{C, s}_t + h_t^{E, s}, \] 
		 where $h^{C, s}_t\in C$ and $\| h_t^{E, s}\|_c \leqslant \varepsilon$. 
		
		In particular, if the above conditions hold for every $r\in \NN$, then  for every $0\leqslant t \leqslant T$ it holds that
		\[ e^{tQ}h = h^{C}_t + h^E_t, \] 
		where $h^{C}_t\in C$ and $\| h^E_t\|_c\leqslant  \varepsilon$.
	\end{theorem}
	\begin{proof}
		Using Corollary~\ref{cor-matrix-series-cone}, we calculate 
		\[ p_s(tQ)h = \sum_{k = 0}^s \lambda_k p_k(TQ)h = \sum_{k = 0}^s \lambda_k (h^{C, k}_T + h_T^{E, k}) = h^{C, s}_t + h_t^{E, s}, \]
		where $h^{C, s}_t := \sum_{k = 0}^s \lambda_k h^{C, k}_T$ and $h^{E, s}_t := \sum_{k = 0}^s \lambda_k h^{E, k}_T$. 
		Clearly, $h^{C, s}_t = \sum_{k = 0}^s \lambda_k h^{C, k}_T \in C$ and $\| h_t^{E, s} \| = \left\| \sum_{k = 0}^s \lambda_k h_T^{E, k} \right\|_c \leqslant   \sum_{k = 0}^s \lambda_k\| h_T^{E, k} \|_c \leqslant \varepsilon$, where the last inequality is implied by sublinearity of the variational seminorm. 
		The last statement immediately follows. 
	\end{proof}
	
	\subsection{Checking the normal cone inclusion}
	In \cite{skulj:15AMC}, verification whether some $p_r(tQ)$ belongs to a normal cone $C$ was implemented through the application of linear programming, which is computationally costly. Here we propose a procedure that vastly reduces the number of linear programming routines that need to be executed and replace them with faster matrix methods. Notice again that in the case a normal cone contains a subset of $\gambleset$ that is not linearly independent, a subset generated by an independent subset is only considered. 
	
	Let $M_J$ denote the matrix whose columns are $f_i$ for $i\in J\subset I$ and $1_\states$ as the first column. Here $J$ stands for any $I_{h, k}$. Equation \eqref{eq-nc-decomposition} is equivalent to $M_J\low{\alpha} = h$, where $\low{\alpha}$ denotes the vector of the coefficients $\alpha(i)$ for $i\in J\cup\{ 0 \}$. Now we write $\alpha(i)$ instead of $\alpha_i$ to avoid multiple indices. Due to the assumed linear independence, $M_J$ is reversible and we have that $\low{\alpha}= M_J^{-1}h$. 
	
	Let $\low{\alpha}^0$ be the vector of coefficients such that $M_J\low{\alpha}^0 = h_0$ and $p_r(t)$ be the $r$-th partial sums for some power series. Further, let $\low{\alpha}_r^t$ be such that $M_J\low{\alpha}_r^t = p_r(tQ)h_0$. It is a matter of basic matrix algebra to prove that 
	\begin{equation}\label{eq-Qj}
		\low{\alpha}_r^t = p_r(tM_J^{-1}QM_J)\low{\alpha}_0 = p_r(tQ_J)\low{\alpha}_0.
	\end{equation}
	That is $Q_J := M_J^{-1}QM_J$ is the matrix corresponding to $Q$ in the basis $\{ 1_\states \}\cup\gambleset_{J} $. The vector $p_r(tQ)h_0$ is in the cone $C_J$ iff $\low{\alpha}_r^t$ has all components, except possibly for the first one, non-negative. To avoid unnecessary calculations, one should first check whether $e^{TQ_J}\low{\alpha}^0$ satisfies these requirements. 
	\begin{proposition}
		Let $\low{\alpha}_0$ be a $m$-tuple and $Q_J$ a square matrix defined above. Denote $\low{\alpha}^t_r=p_r(tQ_J)$, where $p_r(t)$ are the $r$-th partial sum polynomials for the Taylor series of the exponential function. Suppose that $\alpha_{r}^T(i)\geqslant 0$ for every $r\geqslant 0$ and $i\in J$. Then $\alpha_{\infty}^t(i)\geqslant 0$ for every $i\in J$ and $0\leqslant t\leqslant T$, where $\low{\alpha}_\infty^t = e^{tQ_J}\low{\alpha}_0$. 
	\end{proposition} 
	\begin{proof}
		The proposition is a direct application of Corollary~\ref{cor-single-cone}. 
	\end{proof}
	The above proposition provides a directly applicable criterion for checking whether the solution of \eqref{eq-BSDE} on some interval is entirely contained in the same normal cone. If the inclusion holds for all normal cones corresponding to rows $\qset_k$, then the exact solution of  \eqref{eq-BSDE} is obtained as $h_T = e^{TQ}h_0$. 
	
	\subsection{Approximate matrix exponential method}\label{ss-amem}
	The solution using the exponential method might sometimes not satisfy the conditions of the previous subsection exactly and can thus for a particular interval partially lie outside the starting normal cone, yet the distance to it might be small enough to ensure that the error is within required bounds. In this subsection we give a theoretical basis for such a use. 
	
	Let $J = I_{h, k}$ for some row index $k$, an initial vector $h$ be given and denote by $C$ the normal cone $N_{\mathcal Q_k}(Q_k)$. Let $h_t^r = M_J \low\alpha_r^t$, where $\low\alpha^t_r$ are as in the previous subsection. We decompose $\low\alpha^t_r$ into $(\low\alpha_r^t)^+$, which is the vector of its positive components and $(\low\alpha_r^t)^+(0) = \low\alpha_r^t(0)$ and $(\low\alpha_r^t)^-$ containing the absolute values of the negative components except for $(\low\alpha_r^t)^-(0) = 0$. We have that $\low\alpha_r^t = (\low\alpha_r^t)^+ - (\low\alpha_r^t)^-$. Hence $h_t^r = h^{C, r}_t + h^{E, r}_t$, where 
	$h^{C, r}_t = M_J (\low\alpha_r^t)^+$ and $h^{E, r}_t = -M_J (\low\alpha_r^t)^-$. Clearly, $h^{C, r}_t \in C$. 
	\begin{theorem}\label{thm-cone-error-eps}
		We assume the notation used above. 
		Let $h\in \RR^m$ and $Q\in \qset$ be given such that $Qh = \low Q h$. Further, suppose that $\| h^{E, s}_T \|_c = \| -M_J (\low\alpha_s^T)^- \|_c \leqslant \varepsilon$ for some $T>0, \varepsilon > 0$ and all $J\in \{ I_{h, k}\colon 1\leqslant k \leqslant m \}$ and $s\in \{ 0, \ldots, r\}$. Then the inequality 
		\[ \left\| Q[h_t^s] - \low Q[h_t^s]\right\| \leqslant \iota(\qset)\varepsilon \] 
		holds for every $t \in [0, T]$ and $s\in \{ 0, \ldots, r\}$, where $p_s(tQ)$ denote the partial sums for the exponential Taylor series. 
		
		In particular, if the above conditions hold for every $r\in \NN$, then 
		\[ \left\| Q[e^{tQ}h] - \low Q[e^{tQ}h]\right\| \leqslant \iota(\qset)\varepsilon. \] 
	\end{theorem}
	\begin{proof}
		Let $C = N_{\qset_k}(Q_k)$ and denote $h_t^s = p_s(tQ)h$. By the assumption, $h_T^{C, s}\in C$ and $h_T^{E, s}$ with $\| h_T^{E, s}\|_c \leqslant \varepsilon$ exist for every $s\in \{ 0, \ldots, r\}$ such that $h_T^s = h_T^{C, s} + h_T^{E, s}$. Hence, by Theorem~\ref{thm-normal-decomposition}, $h_t^{C,s}\in C$ and $h_t^{E,s}$ with $\| h_t^{E,s}\|_c \leqslant \varepsilon$ exist for every $t \in [0, T]$ and every $s\in \{ 0, \ldots, r\}$ such that $h_t^s = h_t^{C,s} + h_t^{E, s}$. 
		
		By the superadditivity of $\low Q_k$, we have that $\low Q_k (h_t^s) \geqslant \low Q_k(h_t^{C,s}) + \low Q_k(h_t^{E, s})$. Using additivity of $Q_k$, we can write
		\begin{align*}
			\| Q_k (h_t^s) - \low Q_k (h_t^s) \| & \leqslant \| Q_k(h_t^{C,s}) + Q_k(h_t^{E, s}) - \low Q_k(h_t^{C,s}) - \low Q_k(h_t^{E, s}) \| \\
			& = \| Q_k(h_t^{E, s}) - \low Q_k(h_t^{E, s}) \| \\
			& \leqslant \iota(\qset_k) \| h_t^{E, s} \|_c\leqslant \iota(\qset) \varepsilon
		\end{align*}
		for every $k\in \{ 1, \ldots, m\}, s\in \{ 0, \ldots, r\}$ and $t\in [0, T]$. It follows directly that
		\begin{equation}
			\| Q (h_t^s) - \low Q (h_t^s) \| = \max_{1\leqslant k\leqslant m} \| Q_k (h_t^s) - \low Q_k (h_t^s) \| \leqslant \iota(\qset) \varepsilon,
		\end{equation}
		which completes the proof. 
	\end{proof}
	The above proposition provides a base for the use of the matrix exponential approximation in the case the solution on an interval is nearly contained in the same normal cone. 

	\subsection{Grid methods}
	In the case where the error produced by the matrix exponential method exceeds the threshold, one can resort to the so-called uniform grid method. Our assumption is this approach would merely be needed on some isolated intervals where the solution $h_t$ rapidly transits between normal cones, not allowing to use the same minimizing $Q$-matrices for a sufficiently long interval. In fact, by allowing the approximate matrix exponential method in addition to the exact version, our testing showed that most often the use of uniform approach is not needed. Nevertheless, it is the most often described technique in the literature. 
	
	All grid methods divide the interval $[0, T]$ into subintervals $[t_i, t_{i+1}]$, where $0=t_0 < t_1 < \dots < t_n = T$. Then the solutions of equation \eqref{eq-BSDE} are approximated on the individual intervals. The widths of those subintervals are chosen so that the total error remains within required bounds. We will turn back to the error estimation later. Now we suppose the intervals are of the appropriate widths, either uniform or adaptive. Then still two distinct variations of the method are implemented. The first one was first proposed in our earlier paper \cite{skulj:15AMC} and approximates the solution at time $t_{k+1}$ given the one at time $t_k$ as $\hat h_{t_{k+1}}=e^{(t_{k+1}-t_k)Q_k}\hat h_{t_k}$, where the matrix $Q_k$ is such that $Q_kh_{t_k} = \low Qh_{t_k}$. The approach proposed by \cite{krak2017imprecise, erreygers2017imprecise} calculates the new solution as $\hat h_{t_{k+1}} = \left(I+(t_{k+1}-t_k)Q_k \right)\hat h_{t_k}$, using the same way to find the matrix $Q_k$. The latter approach is in fact an approximate version of the former one using the first order Taylor polynomial approximation. The advantage of the first approach is in that the approximate solution $\hat h_t$ does satisfy the differential equation $\frac{\dd}{\dd t}\hat h_t = Q_t h_t$, at every time $t$ for some $Q_t\in \qset$, which in turn ensures that $\hat h_t \geqslant h_t$, where $h_t$ is the true solution. The advantage of the second method is in its computational simplicity, which makes it faster to apply. As we will see later, the error generated by the use of both methods is virtually identical. 

	\section{Error estimation}\label{s-ee}
	In this section we estimate the maximal possible error of the approximation $\hat h_t$ of the exact solution $h_t$ of equation \eqref{eq-BSDE}, employing one of the described methods. We will assume that $\hat h_t \geqslant h_t$ and that it satisfies the equation 
	\begin{equation}\label{eq-approximate-de}
		\frac{\dd h_t}{\dd t} = Q_t h_t,
	\end{equation}
	where $Q_t\colon [0, T]\to \qset$ is some piecewise constant map. We require this property for the sake of simplicity and because actually all the described methods indeed produce such functions. In fact, as far as polyhedral sets of $Q$-matrices are concerned, this property indeed holds, as the matrix minimizing the expression $Qh$ is constant as long as $h$ remains in its normal cone. Note however, that the grid method using the linear approximation $\hat h_{t_{k+1}} = \left(I+(t_{k+1}-t_k)Q_k \right)\hat h_{t_k}$ does not necessarily satisfy equation \eqref{eq-approximate-de}, yet it turns out that the error produced is of similar magnitude.  
	
	\subsection{General error bounds}
	Denote by $P_{\Delta t}$ the linear operator mapping $h$ to the solution of the differential equation \eqref{eq-BSDE} at time $t+\Delta t$ with the initial value at $h_t = h$. We can write $P_{\Delta t}h_t = h_{t+\Delta t}$. Moreover, we will denote by $\hat P_{\Delta t}$ the operator that maps $h_t$ to the approximation $\hat h_{t+\Delta t}$. 
	\begin{proposition}
		$\| P_{\Delta t} \| \leqslant 1$ and $\| \hat P_{\Delta t} \| \leqslant 1$.
	\end{proposition}
	\begin{proof}
		$P_{\Delta t}$ is a lower transition operator, known to have the norm bounded by 1 (see e.g. \cite{krak2017imprecise}), and $\hat P_{\Delta t}$ is a precise transition operator and therefore also has norm bounded by 1. 
	\end{proof}
	Denote by $E_t$ the error of an approximation $\hat h_t$. Thus, $E_t = \| \hat h_t - h_t \|$. Our goal is to estimate $E_t$, and prescribe the optimal method of calculation of $\hat h_t$ that ensures $E_T \leqslant E$, where $E$ is a given maximal allowed error. 
	\begin{proposition}\label{prop-error-carry-on}
		Let $\hat h_t$ be an approximate to the minimal solution $h_t$ of equation \eqref{eq-BSDE} such that $\hat h_t \geqslant h_t$. Moreover, let $\tilde h_{t+\Delta t}$  be the minimal solution of equation \eqref{eq-BSDE} with the initial value in $t$ taken to be the approximate value $\hat h_t$. Then $\| \tilde h_{t+\Delta t}- h_{t+\Delta t} \| = \| P_{\Delta t}\hat h_t - P_{\Delta t}h_t \| \leqslant \| \hat h_t - h_t \|$. 
	\end{proposition}
	\begin{proof}
		The operator $P_{\Delta t}$ is a lower transition operator. Now let $\tilde P_{\Delta t}$ be its corresponding upper transition operator. It is a well-known property of superlinear operators that $P_{\Delta t}(f + g) \leqslant P_{\Delta t}f + \tilde P_{\Delta t}g$, whence $P_{\Delta t}\hat h \leqslant P_{\Delta t}h + \tilde P_{\Delta t}(\hat h-h)$. Moreover, since  $P_{\Delta t}h_t \leqslant P_{\Delta t}\hat h_t$, it follows that  $\| P_{\Delta t}\hat h_t - P_{\Delta t}h_t \| \leqslant \| \tilde P_{\Delta t} \| \| \hat h_t - h_t \| \leqslant \| \hat h_t - h_t \|$, using $\| \tilde P_{\Delta t} \| = 1$.
				
	\end{proof}
	Using an approximation method, the obtained estimate at time $t+\Delta t$ is not $\tilde h_{t+\Delta t}$ but instead an approximation $\hat h_{t+\Delta t}$, which in addition to the error $E_t$ contains an additional error due to the approximation method used. Let $\metherr_{\Delta t}$ denote the error of the method on the interval $\Delta t$. That is, $\metherr_{\Delta t} = \| \hat P_{\Delta t} \hat h_t - P_{\Delta t} \hat h_t \|$. The following proposition holds. 	 
	\begin{proposition}\label{prop-eror-method-general}
		Let $E_t = \| \hat h_t - h_t\|$ for every $t\in [0, T]$ and let $\metherr_{\Delta t}$ denote the error produced by an approximation method on an interval of width $\Delta t$. Then $E_{t+\Delta t} \leqslant E_t + \metherr_{\Delta t}$, which we can rewrite into $\Delta E_t  \leqslant \metherr_{\Delta t}$. 
	\end{proposition}
	\begin{proof}
		We have 
		\begin{align*}
			E_{t+\Delta t} & = \| \hat h_{t+\Delta t} - h_{t+\Delta t} \| \\
			 & \leqslant \| \hat P_{\Delta t} \hat h_t - P_{\Delta t}h_t \| \\
			 & \leqslant \| \hat P_{\Delta t} \hat h_t - P_{\Delta t}\hat h_t \|  +  \| P_{\Delta t} \hat h_t - P_{\Delta t}h_t \| \\
			 & \leqslant \metherr_{\Delta t} + \|  P_{\Delta t} \| \| \hat h_t - h_t \| \\
			 & \leqslant \metherr_{\Delta t} + E_t,
		\end{align*}
		where we used $\|  P_{\Delta t} \| = 1$.
	\end{proof}
	The above proposition could be interpreted as an estimate of the total error that results from the error in initial solution $E_t$ and the error of the method $\metherr_{\Delta t}$. 
	
	\subsection{Error estimation for a single step}
	Within a single approximation step we calculate the solution $\hat h_{t+\Delta t}$ based on the approximation $\hat h_t$. For the purpose of error estimation, we will set $t=0$ and $\Delta t = T$. Moreover, we will assume the initial solution is exact, because otherwise, the initial error is merely added to the error of the method as shown in the previous section. Thus, the initial value is set to $h_0 = h$. 
	
	Now assume we have the estimation of the form $\hat h_t = e^{tQ}h$ for $t\in [0, T]$, where $Q h = \low Q h$. Our goal is to bound the norm of the difference $\hat h_{T}-h_T$, where $h_T$ is the exact solution of equation \eqref{eq-BSDE} with initial condition $h_0 = h$. 
	
	Let us introduce some more notation. Let $E_t = \| \hat h_t - h_t\|$ represent the error of the approximation. By definitions, we have that $\dd\hat h_t = Q\hat h_t\dd t$ and $\dd h_t = \low Qh_t \dd t$. Also recall the notation introduced in Section~\ref{ss-amem}. 
	\begin{theorem}\label{thm-expm-error}
		Let $h \in \RR^m$ be given and the matrix $Q$ be such that $Qh = \low Qh$. Suppose that for some $T>0$ and $\varepsilon >0$ we have that $\| h^{E, r}_T \|_c = \| M_J (\low\alpha_r^T)^- \|_c \leqslant \varepsilon$ for every $r \in \NN \cup\{ 0 \}$. Then 
		\begin{equation}
			E_t \leqslant \left(e^{\| \qset \| t}-1\right)\frac{\iota(\qset)}{\| \qset \|}\varepsilon \leqslant 2\left(e^{\| \qset \| t}-1\right)\varepsilon 
		\end{equation}
		 for every $0\leqslant t \leqslant T$. 
	\end{theorem}
	\begin{proof}
		We first make some estimates related to the error $E_t$. Let $\hat h_t = e^{tQ}h$. By Theorem~\ref{thm-cone-error-eps}, we have that $\| Q\hat h_t - \low Q\hat h_t\| \leqslant \iota(\qset) \varepsilon$. 
		
		Next, it follows by the basic properties of vector norms that
		\begin{align*}
			\| \dd\hat h_t - \dd h_t \| & = \| \hat h_{t+\dd t} - \hat h_t - (h_{t+\dd t} - h_t) \| \\
			& \geqslant \| \hat h_{t+\dd t} - h_{t+\dd t} \| - \| \hat h_t - h_t \| \\
			& = E_{t+\dd t} - E_t = dE_t. 
		\end{align*}
		We also have that 
		\begin{align*}
			\| \dd\hat h_t - \dd h_t \| & = \| Q\hat h_t - \low Q h_t \|\dd t \\
			& \leqslant \| Q\hat h_t - \low Q\hat h_t + \low Q\hat h_t - \low Q h_t \|\dd t \\
			& \leqslant \| Q\hat h_t - \low Q\hat h_t \| \dd t + \| \low Q\hat h_t - \low Q h_t \|\dd t \\
			& \leqslant \iota(\qset)\varepsilon \dd t+ \| \qset \| \| \hat h_t- h_t \|\dd t \\
			& = \iota(\qset)\varepsilon\dd t + \| \qset \| E_t\dd t \\
			& \leqslant 2\|\qset\|\varepsilon\dd t + \| \qset \| E_t\dd t
		\end{align*}
		Combining the above inequalities gives
		\begin{equation}\label{eq-de-error}
			\frac{\dd E_t}{\dd t} \leqslant \iota(\qset)\varepsilon + \| \qset \| E_t \leqslant 2\|\qset\|\varepsilon + \| \qset \| E_t.
		\end{equation}
		The maximal error is thus bounded by the solution of the differential equation
		\begin{equation}
			\frac{\dd E_t}{\dd t} = \iota(\qset)\varepsilon + \| \qset \| E_t .
		\end{equation}
		Under the initial condition $E_0=0$, the solution is $E_t =  (e^{\| \qset \| t}-1)\frac{\iota(\qset)}{\| \qset \|}\varepsilon\leqslant 2\left(e^{\| \qset \| t}-1\right)\varepsilon$, and this completes the proof. 
	\end{proof}
	
	\subsection{Upper bound for the error}\label{ss-ube}
	Consider again the operator $e^{tQ}$ acting on vector $h$, which is by definition equal to
	\begin{equation}
		e^{tQ}h = h + \sum_{k=1}^\infty \frac{(tQ)^k}{k!}h =: h + h_t^E.
	\end{equation}
	We now estimate the variational seminorm of $h_t^E$ as a function of $t$, using sublinearity of the seminorm and repeated application of Proposition~\ref{prop-q-matrix-variational}:
	\begin{equation}
		\varepsilon(t) = \| h_t^E \|_c \leqslant \| h \|_c \sum_{k=1}^\infty \frac{(\| tQ\|)^k}{k!} \leqslant \| h \|_c (e^{t\|Q\|}-1), 
	\end{equation}
	which is the worst case estimate for the norm of the component lying outside the normal cone $N_\qset(Q)$. With a small straightforward modification of the differential equation \eqref{eq-de-error}, we obtain 
	\begin{equation}\label{eq-de-error-general}
			\frac{\dd E_t}{\dd t} \leqslant \iota(\qset)\varepsilon(t) + \| \qset \| E_t \leqslant 2\|\qset\|\| h \|_c (e^{t\|Q\|}-1) + \| \qset \| E_t .
	\end{equation}
	The solution of the above differential inequality is bounded from above by the solution of the corresponding differential equation, which is, subject to $E_0=0$,  
	\begin{equation}\label{eq-one-step-maximal-error}
		E_t = 2\| h \|_c (1-e^{t\| \qset \|}(1-t\| \qset\|)).
	\end{equation}
	
	\subsection{Error estimation for the uniform grid}
	The approximation using the uniform grid method on an interval $[0, T]$ is obtained by dividing the interval into subintervals $[t_i, t_{i+1}]$, where $0=t_0 < t_1 < \dots < t_n = T$. Although the differences $t_{i+1}-t_i$ can be variable in some approaches (see e.g. \cite{erreygers2017imprecise}), we will conveniently assume all distances are equal to $\delta = \frac{T}{n}$. By Proposition~\ref{prop-eror-method-general}, the error at time $t_k$ satisfies the following recursive relation
	\begin{equation}
		E_{t_k} \leqslant E_{t_{k-1}} + \metherr_{t_k-t_{k-1}} = E_{t_{k-1}} + \metherr_{\delta}, 
	\end{equation}
	where $\metherr_\delta$ is the error of the one step method, which by equation \eqref{eq-one-step-maximal-error} satisfies $\metherr_\delta \leqslant 2\| h_{t_{k-1}}\|_c(1-e^{\delta\| \qset \|}(1-\delta\| \qset\|))$. Note, however, that $h_t = P_t h$, where $h$ is the initial value and $P_t$ a transition operator, and therefore $\| h_t \|_c \leqslant \| P_t \| \| h\|_c \leqslant \| h \|_c \leqslant \| h \|$, since $\| P_t\| = 1$ is well-known. This is a very conservative estimate and could be improved using ergodicity properties of the operators $P_t$. The total error on the interval $[0, T]$ is bounded by the sum of the errors on the subintervals, which by \eqref{eq-one-step-maximal-error} is equal to
	\begin{equation}\label{eq-uniform-grid-error}
		E_T \leqslant  2n\| h\|(1-e^{\delta\| \qset \|}(1-\delta\| \qset\|)) = 2n\| h\|\left(1-e^{\frac Tn\| \qset \|}\left(1-\frac Tn\| \qset\|\right)\right).
	\end{equation}
	In \cite{erreygers2017imprecise}, an error estimate for a uniform grid method which uses the approximation of $h_{t} = (I+(t_n-t_{n-1})\low Q)h_{t-1}$, has been found to be 
	\begin{equation}
		E^*_T = \delta^2\| \qset \|^2\sum_{i=0}^{n-1}\| h_{t_i} \|_c.
	\end{equation} 
	In the worst case we have that $\| h_{t_i} \|_c = \| h \|$, where we end up with the estimate 
	\begin{equation}
		E^*_T = n\delta^2\| h \|\| \qset \|^2, 
	\end{equation}
	which is very close to our estimate \eqref{eq-uniform-grid-error}, especially for large $n$. 
	
	Both our error estimate and the one found in \cite{erreygers2017imprecise}, benefit from ergodicity properties, causing diminishing the variational norm of the solution vector function. 
	
	\section{Algorithm and examples}\label{s-ae}
	Based on the theoretical results, we now provide an algorithm for estimating the solution of equation \eqref{eq-BSDE} with given imprecise transition rate matrix $\qset$ and initial value $h$. 	
	\subsection{Parts of the algorithm}
	We will present the version of the algorithm where only the matrix exponential method is used. 
	\subsubsection{Inputs} The following inputs to the algorithm are needed:
	\begin{itemize}
		\item a set of gambles $\gambleset$ is given in terms of an $N\times m$ matrix, where the $i$-th row denotes a gamble $f_i$;
		\item a set of lower transition rates $\low Q$ is also represented in terms of a $N\times m$ matrix, where the $(i, j)$-th entry denotes $[\low Q f_i]_j$;
		\item a gamble $h$ as an $m$-tuple;
		\item time interval length $T > 0$;
		\item maximal allowed error $E$.
	\end{itemize} 
	\subsubsection{Outputs} The algorithm provides an approximation of $h_T$ as an $m$-tuple and $Er$, the maximal bound on the error. Note that the calculated approximation can be more accurate than required. The requirement is that $Er \leqslant E$. 
	\subsubsection{Minimizing matrix} The matrix $Q$ satisfying $Qh = \low Qh$ is found using linear programming. For each $k = 1, \ldots, m$, the following linear programming problem is solved:
	\begin{quote}
		Minimize:
		\begin{align}
			Q_k h &  \label{eq-minimizing-nc-objective}
			\intertext{subject to }
			Q_k f_i & \geqslant \low Q_k f_i \label{eq-minimizing-nc-nonnegativity} \\
			Q_k 1_\states & = 0. 
		\end{align} 
	\end{quote}
	The matrix $Q$ consists of the resulting rows $Q_k$. 
	\subsubsection{Identification of the normal cones} For each row $k = 1, \ldots, m$, we identify the index set $I_k = \{ i\in \{ 1, \ldots, N\} \colon Q_kf_i = \low Q f_i \}$. Further, we calculate:
	\begin{itemize}
		\item a non-negative linear combination $\sum_{i\in I_k} \alpha_i f_i = h$ and
		\item if $|I_k|>m-1$, a non-trivial linear combination $\sum_{i\in I_k} \beta_i f_i = 0$.
	\end{itemize}
	Based on the above combinations, a gamble $f_i$ is eliminated as described in Section~\ref{ss-fliplc}. The above steps are repeated until $\gambleset_{I_{k}}$ becomes linearly independent. If needed, the set is completed to a basis with some of the remaining elements of the cone basis. 
	The final output is a linearly independent set $\gambleset_{I_{k}}$ and a collection of coefficients $\low \alpha = (\alpha_0, \ldots, \alpha_{m-1})$ for every row $k$. 
	In the case where some normal cones coincide for different rows, the duplicates are removed. 
	\subsubsection{Finding a feasible interval}\label{sss-ffi} In general, the application of the matrix exponential method on the entire interval $[0, T]$ is infeasible. Hence, we need to find a subinterval $[0, T']$ where the error is within required bounds. As by Proposition~\ref{prop-eror-method-general} the errors are sequentially added to the initial error, we require that the added part of the error $E^m_{T'}$ is smaller than the proportional part of the maximal allowed error: $Er \leqslant ET'/T$. 
	This error estimate is calculated using Theorem~\ref{thm-expm-error}. Its estimation first requires the assessment of $\varepsilon$, which is obtained by applying Theorem~\ref{thm-cone-error-eps}, as $\varepsilon = \min_{J\in I_k}\| M_J(\alpha_s^{T'})^- \|_c$. 
	The initial estimate of the interval length is obtained, using the linear approximation of $e^{tQ}h\approx h + tQh$, to be the maximal $t$ such that $\alpha_0 + tQ_J \alpha_0 \geqslant 0$ (see \eqref{eq-Qj}). If $\alpha_0$ happens to have zero elements, then the above expression may have negative coefficients even for very small values of $t$, in which case we just try with a minimal initial interval, specified as a parameter of the algorithm. 
	
	In case the initial interval yields too large estimated error, the interval is halved until reaching the required error size. Since the estimated error size is at most as large as with the grid methods reported in \cite{erreygers2017imprecise, skulj:15AMC}, the process eventually produces a feasible interval. 
	\subsubsection{Iterative step} Once a feasible interval length $dt$ is found, the new initial solution is set to $h_{dt} = e^{dtQ}h$. The remaining time interval then reduces to $T-dt$. The maximal allowed error is updated to $E-Er$, where $Er$ is the evaluated maximal error of the applied method. 

	Algorithm~\ref{alg-approx} illustrates the main steps of the approximation of the solution using our method. 
	 \begin{algorithm}\label{alg-approx}
		\caption{Function: approximate $h_T$}
		\begin{algorithmic}[1]
			\Require $\gambleset, \low{Q}, h, T, E$
			\Ensure $h_T, maxErr$ \Comment{solution at time $T$, error estimate}
			\State $t_s = 0, t_e = T$ \Comment{start and end time points}
			\State $maxErr = 0$ 
			\State $nq = \| \qset \|, io = \iota(\qset)$
			\While{$t_s < t_e$}
				\State $Q = \arg\min_{Q\in \mathcal Q}Qh$
				\For {$k=1, \ldots, m$}
					
					\State $(I_k) =$ normalCone$(h, Q_k, \qset_k)$
					\State\Comment{find the basis of the normal cone for $k$-th row}
					
					\State $(I^i_k, ind_k) =$ reduceToIndependent$(I_k, h)$
					\State\Comment{reduce to independent set and find linear combination equal $h$}
				\EndFor
				\State $dt =$ min(initialInterval$(I, ind), t_e-t_s)$
				\State \Comment{try initial interval based on the linear approximation} 
				\Repeat 
					\State $\varepsilon=$ estimateEpsilon$(I, ind)$
					\State $Err = (e^{nq\,t}-1)\frac{io}{nq}\varepsilon$ \Comment{estimated error}
					\State $Ea = E\cdot dt/T$ \Comment{maximal allowed error}
					\If {$(Err > Ea)$} \State 	$dt = dt/2$ \EndIf
				\Until {$Err \leqslant  Ea$}
				\State $h = e^{dt\, Q}h$\Comment{new solution} 
				\State $maxErr = maxErr + Err$ \Comment{total error}
				\State $E = E-Err$ \Comment{the remaining allowed error}
				\State $t_s = t_s + dt$ \Comment{new starting point}
			\EndWhile   
			\State \Return $h_T = h, maxErr$
		\end{algorithmic}
	\end{algorithm}
	\subsection{Examples}
	In our first example we demonstrate the use of the method for a case where the solution remains in a single normal cone for the entire interval. 
	\subsubsection{Example 1} 
	Let $\states$ be a set of 3 states, which we denote by $1, 2, 3$. We consider a set $\qset$ of $Q$-matrices which is given by the constraints of the form $\low Q_i (1_A)$ for all non-trivial subsets in $\states$. As in addition we want to ensure that the representing gambles $f$ all satisfy $\sum_{k\in \states} f_k = 0$ and to be of norm equal to 1, we instead use the following six representing gambles
	\begin{align*}
		f_1 & = (-1, 1/2, 1/2) & f_2 & = (1/2, -1, 1/2) & f_3 & = (-1/2, -1/2, 1) \\
		f_4 & = (1/2, 1/2, -1) & f_5 & = (-1/2, 1, -1/2) & f_6 & = (1, -1/2, -1/2).
	\end{align*} 
	Let the set $\qset$ be specified via the following constraints:
	\begin{equation}
		L = \left( 
		\begin{array}{rrrrrr}
			0.76 & -0.69 & 0.15 & -0.24 & 0.60 & -0.92 \\
			-0.99 & 1.21 & 0.30 & -0.39 & -1.37 & 0.90 \\
			-0.24 & -0.54 & -0.76 & 0.61 & 0.45 & 0.15 \\
		\end{array}
		\right).
	\end{equation}
	The elements of the above matrix denote the lower bounds $l_{ki} = \low Q_k(f_i)$. Now $\qset$ is the set of all $Q$-matrices $Q$ satisfying, for every $i=1, \ldots, 6$, $Qf_i \geqslant L^i$, which denotes the $i$-th column of $L$. Given an initial gamble $h = (-0.7, 1.7, -1)$ we calculate the solution of equation \eqref{eq-BSDE} satisfying $h_0 = h$ on the interval $[0, 1]$. 
	
	We try finding as large as possible an interval where $h_t$ is in the same normal cone of $\qset$ as $h$. The matrix $Q$ minimizing $Qh$ over $\qset$ is found to be
	\begin{equation*}
		Q = \left( 
		\begin{array}{rrr}
			\num{-0.56} & \num{0.46} & \num{0.1} \\
			\num{0.606667} & \num{-0.80667} & \num{0.2} \\
			\num{0.146667} &\num{ 0.36} & \num{-0.50667} \\ 
		\end{array}
		\right).
	\end{equation*} 
	All normal cones $N_{\qset_k}(Q_k)$ are spanned by the same set of gambles $\{ f_4, f_5, 1_\states \}$. Specifically, we have that $h = 1.6f_4 + 0.2f_5$. This is of course due to the fact that we restricted the space of the gambles to the set where the sum of components for each one of them is zero. We cannot expect this for all further $h_t$, whence the constant $1_\states$ will in general appear in the linear combinations forming $h_t$.
	
	Thus, we have the initial vector of coefficients $\alpha_0 = (1.6, 0.2, 0)$ of $h$ in the basis $\mathcal B = (f_4, f_5, 1_\states)$. The preliminary analysis based on the first order Taylor approximation as described in Section~\ref{sss-ffi} suggests that the initial time interval where the matrix exponential method could be applied is the interval $[0, T]$ with $T=\num{0.773941371859648}$. To confirm this interval, all vectors $p_r(TQ)h$ must be contained in the cone generate by non-negative linear combinations of $\mathcal B$, except for the constant. According to the procedure described in Section~\ref{ss-camem}, we find the matrix $Q_J$ which corresponds to the operator $Q$ in the basis $\mathcal B$, which we obtain as 
	\begin{equation*}
		Q_J = M_J^{-1}QM_J = \left( 
		\begin{array}{rrr}
			\num{-1.26667} & \num{-0.1} & \num{0} \\
			\num{0.1} & \num{-0.60667} & \num{0} \\
			\num{0.006667} & \num{0.103333} & \num{0} 
		\end{array}
		\right),
	\end{equation*}
	with $M_J$ being the matrix with elements of $\mathcal B$ as columns. Checking whether $p_r(TQ)h$ is contained in the same cone, directly translates to checking whether $\low\alpha^T_n = p_r(TQ_J)\alpha_0$ has non-negative components corresponding to $f_4$ and $f_5$, that is, in the first two places. The resulting sequence of coefficients is (rounded to two decimals): 
	\begin{align*}
		\low \alpha^T_1 & = (0.016, 0.230, 0.024) & \low\alpha^T_2 & = (0.791, 0.162, 0.021) & \low\alpha^T_3 & = (0.540, 0.192, 0.021) \\
		\low\alpha^T_4 & = (0.601, 0.184, 0.021) & \low\alpha^T_5 & = (0.589, 0.186, 0.021) & \low\alpha^T_\infty & = (0.591, 0.185, 0.021).
	\end{align*} 
	All coefficients $\low\alpha^T_n$ for $n > 5$ lie in the neighbourhood of the limit values $\low\alpha^T_\infty$, and are certainly positive. Every partial sum $p_r(1\cdot Q)h$ therefore belongs to the same normal cone as $h$ and so do all $h_t$ for $t\in [0, T]$, as follows by Corollary~\ref{cor-single-cone}. The solution $h_T = e^{T\cdot Q}h = (-0.182, 0.704, -0.460)$ is therefore the exact solution of the equation \eqref{eq-BSDE} on this interval. Two more steps, similar to this one, are needed to obtain the result $h_1 = (\num{-0.107789092019201}, \num{0.552242160179236}, \num{-0.366297008130663})$.
	
	In this example, the power of the new method is fully demonstrated. First, only three optimization steps needed. For comparison we estimate the required number of steps if the uniform grid method \cite{erreygers2017imprecise} were employed. By the error estimate provided in their paper, $\delta^2\| \qset \|^2\sum_{i=0}^{n-1}\| h_{t_i} \|_c = \frac{1}{n^2}\| \qset \|^2\sum_{i=0}^{n-1}\| h_{t_i} \|_c\leqslant \varepsilon = 0.001$ is required. The norms $\| h_{t_i} \|_c$ are bounded from below using the contraction nature of the transition operators, whence we can deduce that $\| h_{t_i} \|_c \geqslant \| h_1 \|_c = 0.45$. The norm $\| \qset \|$ is bounded by $1.82$. Based on these estimates, the number of required iterations would be at least $1\,490$. Applying our method does bring some additional tasks to be performed, yet these tasks in total contribute much less to the time complexity than the optimization steps. 
	
	Second, knowing that the solution lies in the same normal cone, guarantees not only that the result is accurate up to the maximal allowed error, but also that it is the exact solution. Using the approximate operators $(I+ \frac Tn \low Q)^n$, the best we can get are approximations.  
	\subsubsection{Example 2}
		In our second example we revise example in \cite{troffaes2015using}, Section 3.4. In this example the states denote failures in a power network, and the transitions arise from the repair rates. The imprecise transition rate matrix there is given as a pair of a lower and upper transition rate matrices:
		\begin{align}
			Q_L = 
			\begin{bmatrix}
				-0.98 & 0.32 & 0.32 & 0.19 \\
				730 & -1460.61 & 0 & 0.51\\
				730 & 0 & -1460.61 & 0.51\\
				0 & 730 & 730 & -2920
			\end{bmatrix}
			\\
			Q_U = 
			\begin{bmatrix}
				-0.83 & 0.37 & 0.37 & 0.24 \\
				1460 & -730.51 & 0 & 0.61 \\
				1460 & 0 & -730.51 & 0.61 \\
				0 & 1460 & 1460 & -1460
			\end{bmatrix}, 
		\end{align}
		where we can simply take
		\begin{multline}
			\mathcal{Q} = \left[Q_L, Q_U\right]
			=\Big\{Q\colon Q_{L, k} \leqslant Q_k \leqslant Q_{U, k}, 
			\forall 1\leqslant k\leqslant m,\,\sum_{l=1}^mQ_{kl}=0\Big\}
		\end{multline}
		In the original paper, bounds for the long-term distribution were estimated, yet without a clear idea how to estimate the error bounds. 
		
		It was observed, however, that the uniform grid with as little as 80 subintervals was sufficient to obtain a sufficiently accurate result on the interval $[0, 0.02]$, which turned to be sufficient for the process to reach the limit distribution. The error estimates employing the methods at hand predicted significantly larger errors than observed. 
		
		The bounds for the limit distributions were found to be
		\begin{equation}
			\underline\pi=
			\begin{bmatrix}
				\num{9.9849486e-01} \\
				\num{2.6229302e-04} \\
				\num{2.6229302e-04} \\
				\num{6.5126517e-05}
			\end{bmatrix}
			\quad
			\overline\pi=
			\begin{bmatrix}
				\num{9.9936674e-01} \\
				\num{7.252061e-04} \\
				\num{7.252061e-04} \\
				\num{1.6469619e-04}
			\end{bmatrix}
		\end{equation}
		To calculate the lower transition probability $\lpr_t(\{ i | j\})$ we first find the solution $h_t$ of \eqref{eq-BSDE} for $h_0 = 1_{\{ i \}}$ and take its $j$-th component $[h_t]_j$. To calculate the upper probability, we take $h_0 = -1_{\{ i \}}$ and then set $\up P_t(\{ i | j\})=-[h_t]_j$. For a sufficiently large time interval and a convergent chain, all components of $h_t$ became more and more similar and in our case they denote the limit lower respectively upper probabilities. 
		
		We repeated the calculations utilizing our method, setting the maximal allowed error to $0.001$ and the time interval to $[0, 1]$, that is clearly more than sufficient to ensure convergence. The method produced identical results on the lower and upper bounds, with the number of required iterations for each value varying between 30 and 40. Our method therefore confirms the validity of the results in the original paper, which does not contain a rigorous proof. 
		
		\section{Concluding remarks}
		The method presented in this paper provides a promising alternative to the existing methods  for approximating the solutions of the imprecise generalization of the Kolmogorov backward differential equation on finite intervals. The primary achievement is that the approach of matrix exponentials no longer needs to be combined with the grid methods. This is predominantly thanks to the introduction of the approximate version of the exponential method and considerably improved error estimation.

		As presented, our analysis is limited to finite intervals; however, with some adaptations, it could be employed for finding the limit distributions as well. A step into this direction is demonstrated in our second example, where the obtained solution is effectively the limit distribution. The convergence manifests in the solutions becoming close to a constant vector. Put differently, the difference to a constant tends to zero, which is taken into account by the error estimates. It is a matter of further work to formalize this into a comprehensive method for finding long term distributions.

	\section*{Acknowledgement}
	The author acknowledges the financial support from the Slovenian Research Agency (research core funding No. P5-0168).
	
	 \bibliographystyle{splncs04}
	\bibliography{references_all}
\end{document}